\documentclass{louloupart1}
\usepackage{newtxtext,newtxmath}
\usepackage{pinlabel}
\usepackage{graphicx}
\usepackage[all]{xy}
\usepackage{tikz}

\title[A cube complex for virtual Artin groups]{A cube complex for virtual Artin groups}
\author[J. G\'alvez Mateos]{Jos\'e G\'alvez Mateos}
\givenname{Jos\'e}
\surname{G\'alvez Mateos}
\address{Jos\'e G\'alvez Mateos, Instituto de Matem\'aticas de la Universidad de Sevilla (IMUS) and Departamento de \'Algebra, Universidad de Sevilla, Spain.}
\email{jgalvez7@us.es}
\author[L. Paris]{Luis Paris}
\givenname{Luis}
\surname{Paris}
\address{Luis Paris, Universit\'e Bourgogne Europe, CNRS, IMB, UMR 5584, 21000 Dijon, France.}
\email{lparis@u-bourgogne.fr}

\newtheorem{thm}{Theorem}[section]
\newtheorem{lem}[thm]{Lemma}
\newtheorem{prop}[thm]{Proposition}
\newtheorem{corl}[thm]{Corollary}

\theoremstyle{definition}

\theoremstyle{definition2}
\newtheorem*{acknow}{Acknowledgments}
\newtheorem*{defin}{Definition}
\newtheorem*{rem}{Remark}
\newtheorem*{fund}{Funding}

\numberwithin{equation}{section}

\makeatletter
\renewcommand{\thefigure}{\ifnum \c@section>\z@ \thesection.\fi
 \@arabic\c@figure}
\@addtoreset{figure}{section}
\makeatother

\begin{document}

\def\N{\mathbb N} \def\SS{\mathcal S} \def\TT{\mathcal T}
\def\VA{{\rm VA}} \def\KVA{{\rm KVA}} \def\PP{\mathcal P}
\def\DD{\mathcal D} \def\sph{{\rm sph}} \def\free{{\rm free}}
\def\Link{{\rm Link}} \def\EE{\mathcal E} \def\CC{\mathcal C}
\def\Cub{{\rm Cub}} \def\Fac{{\rm Fac}} \def\UU{\mathcal U}
\def\GG{\mathcal G} \def\id{{\rm id}} \def\LL{\mathcal L}
\def\E{\mathbb E} \def\VArt{{\rm VArt}} \def\R{\mathbb R}
\def\GL{{\rm GL}} \def\AA{\mathcal A} \def\XX{\mathcal X}
\def\BB{\mathcal B}


\begin{abstract}
Let $\Gamma$ be a Coxeter graph, and let $S$ be its vertex set.
An elemental complex over $\Gamma$ is an abstract simplicial complex $\mathcal D$ on $S$ that contains every pair $\{s, t\}$ with $s \neq t$ and $m_{s,t} \neq \infty$.
To each elemental complex $\mathcal D$, we associate a cube complex $\Xi_{\mathcal D}$ on which the virtual Artin group ${\rm VA} [\Gamma]$ acts.
We prove that $\Xi_{\mathcal D}$ is connected, simply connected, and locally regular; that the action of ${\rm VA}[\Gamma]$ on $\Xi_{\mathcal D}$ is regular, cocompact, and by isometries; and that the stabilizers of cells are parabolic subgroups.
We finally prove that $\Xi_{\mathcal D}$ is CAT(0) if and only if ${\mathcal D}$ is a flag complex.



\end{abstract}

\maketitle


\section{Introduction}\label{sec1}

Let $S$ be a finite set.
A \emph{Coxeter matrix} on $S$ is a symmetric square matrix $M=(m_{s,t})_{s,t \in S}$, indexed by the elements of $S$, such that $m_{s,s}=1$ for every $s \in S$, and $m_{s,t} = m_{t,s} \in \N_{\ge 2} \cup \{ \infty\}$ for all distinct $s,t \in S$.
Such a matrix is usually represented by a labeled graph $\Gamma$, called the \emph{Coxeter graph} of $M$.
This graph is defined as follows.
The vertex set is $S$; two distinct vertices $s,t \in S$ are connected by an edge if $m_{s,t} \ge 3$, and this edge is labeled by $m_{s,t}$ whenever $m_{s,t} \ge 4$.
From now on, we fix a Coxeter graph $\Gamma$, and we denote by $M = (m_{s,t})_{s,t \in S}$ its associated Coxeter matrix.

Given two letters $a$ and $b$, and an integer $m \ge 2$, we denote by $(a,b)_m$ the alternating product $\cdots bab$ of length $m$.
The \emph{Artin group} of $\Gamma$, denoted by $A [\Gamma]$, is defined by the following presentation
\[
A [\Gamma] = \langle S \mid (s,t)_{m_{s,t}} = (t,s)_{m_{s,t}} \text{ for } s,t \in S,\ s \neq t \text{ and } m_{s,t} \neq \infty \rangle \,.
\]
The \emph{Coxeter group} of $\Gamma$, denoted by $W [\Gamma]$, is the quotient of $A [\Gamma]$ by the relations $s^2=1$ for all $s \in S$.

Let $\SS = \{ \sigma_s \mid s \in S\}$ and $\TT = \{ \tau_s \mid s \in S \}$ be two sets in bijection with $S$.
The \emph{virtual Artin group} of $\Gamma$, denoted by $\VA [\Gamma]$, is the group defined by the presentation with generating set $\SS \sqcup \TT$ and defining relations:
\begin{itemize}
\item[(V1)]
$\tau_s^2=1$, for every $s \in S$;
\item[(V2)]
$(\sigma_s, \sigma_t)_{m_{s,t}} = (\sigma_t, \sigma_s)_{m_{s,t}}$ and $(\tau_s, \tau_t)_{m_{s,t}} = (\tau_t, \tau_s)_{m_{s,t}}$ for all distinct $s,t \in S$ such that $m_{s,t} \neq \infty$;
\item[(V3)]
$(\tau_s, \tau_t)_{m_{s,t}-1} \sigma_s = \sigma_x (\tau_s, \tau_t)_{m_{s,t}-1}$ for all distinct $s,t \in S$ such that $m_{s,t} \neq \infty$, where $x=s$ if $m_{s,t}$ is even, and $x=t$ if $m_{s,t}$ is odd.
\end{itemize}

Coxeter groups were introduced by Tits in a manuscript \cite{Tit13} written in the late 1950s but published only in 2013.
However, most of its content appeared in the book by Bourbaki \cite{Bou68}, which is the foundational work on Coxeter group theory.
These groups have been extensively studied from many different perspectives since their introduction.
Artin groups were introduced by Tits \cite{Tit66} as extensions of Coxeter groups.
However, their systemic study began in the 1970s with the work of Brieskorn \cite{Bri71, Bri73}, Brieskorn--Saito \cite{BS72}, and Deligne \cite{Del72}.
In recent years, they have attracted renewed interest, particularly in geometric group theory.
Virtual Artin groups were recently introduced by Bellingeri, Thiel, and the second author in \cite{BPT23}.
Their main examples are virtual braid groups (see \cite{Kau99,KL04,Kam07}).
The virtual Artin group of a Coxeter graph $\Gamma$ contains both the Artin group and the Coxeter group associated with $\Gamma$, as well as the action of the latter on its root system (see \cite{BPT23} for further details).

Let $V$ be a set.
We denote by $\PP(V)$ the power set of $V$.
Recall that an \emph{abstract simplicial complex} on $V$ is a subset $\Sigma \subseteq \PP (V)$ satisfying the following conditions:
\begin{itemize}
\item[(a)]
$\emptyset \in \Sigma$, and $\{ x\} \in \Sigma$ for every $x \in V$;
\item[(b)]
every element of $\Sigma$ is finite;
\item[(c)]
for all $\AA, \BB \in \PP (V)$, if $\AA \subseteq \BB$ and $\BB \in \Sigma$, then $\AA \in \Sigma$.
\end{itemize}
The \emph{geometric realization} of an abstract simplicial complex $\Sigma$ on $V$ is denoted by $\Delta (\Sigma)$.
We refer to \cite[Definition 2.27]{Koz08} for its definition.
A \emph{geometric simplicial complex} is the geometric realization of an abstract simplicial complex.

\begin{rem}
In this paper, we adopt the convention that the empty set is an element of every abstract simplicial complex.
However, the empty set is not regarded as a simplex in a geometric simplicial complex.
\end{rem}

Recall that $S$ denotes the vertex set of the Coxeter graph $\Gamma$.
An \emph{elemental complex} over $\Gamma$ is an abstract simplicial complex $\DD$ on $S$ satisfying the following additional condition:
\begin{itemize}
\item[(d)]
$\{s,t\} \in \DD$ for every pair of distinct vertices $s,t \in S$ such that $m_{s,t} \neq \infty$.
\end{itemize}

Two elemental complexes, denoted by $\DD_{\sph}$ and $\DD_{<\infty}$, respectively, are of particular interest.
They are defined as follows.
For $X \subseteq S$, let $W_X [\Gamma]$ denote the subgroup of $W [\Gamma]$ generated by $X$.
We say that $X$ is of \emph{spherical type} if $W_X [\Gamma]$ is finite.
Then $\DD_{\sph}$ is the set of subsets of $S$ of spherical type.
On the other hand, we say that $X \subseteq S$ is \emph{free of infinity} if $m_{x,y} \neq \infty$ for all $x,y \in X$.
Then $\DD_{<\infty}$ is the set of free of infinity subsets of $S$. 

Recall that an abstract simplicial complex $\Sigma$ on $V$ is called a \emph{flag complex} if, for every finite subset $\AA \subseteq V$, the following equivalence holds:
\[
\big( \forall x,y \in \AA,\ \{x,y \} \in \Sigma\big)\ \Leftrightarrow\
\AA \in \Sigma\,.
\]

Let $\DD$ be an elemental complex.
In \cite{GP12}, a finite-dimensional, locally regular, connected, and simply connected cube complex $\Psi_\DD$ is constructed, on which the Artin group $A [\Gamma]$ acts regularly and cocompactly by isometries.
The construction of $\Psi_\DD$ depends on the choice of $\DD$.
Note that, in general, the action of $A [\Gamma]$ on $\Psi_\DD$ is not geometric, since it is not properly discontinuous.
However, the stabilizers of the faces are parabolic subgroups of $A [\Gamma]$.
One of the main results of \cite{GP12} states that $\Psi_\DD$ is CAT(0) if and only if $\DD$ is a flag complex.
The notions of a locally regular cube complex and of a CAT(0) space will be recalled in Sections~\ref{sec4} and~\ref{sec5}, respectively.

If $\DD=\DD_{\sph}$, then $\Psi_\DD$ is the \emph{Deligne complex} introduced by Charney and Davis in \cite{CD95} in their study of the $K(\pi, 1)$ problem for Artin groups.
If $\DD = \DD_{<\infty}$, then $\Psi_\DD$ is called the \emph{clique-cube complex} in \cite{CMW19}.
The latter is arguably the most interesting case, as it is easy to show that $\DD_{< \infty}$ is a flag complex and, if $\DD$ is a flag complex, then $\DD_{<\infty} \subseteq \DD$.
In other words, $\DD_{<\infty}$ is the smallest elemental flag complex. 
If $\DD_{\sph} = \DD_{<\infty}$, then $\Gamma$ is said to be of \emph{FC type}.
This equality plays a crucial role in \cite{CD95} for proving that $A[\Gamma]$ satisfies the $K(\pi, 1)$ conjecture whenever $\Gamma$ is of FC type.

In this paper, we extend the construction of Godelle and the second author \cite{GP12} to virtual Artin groups.
More precisely, we construct a cube complex $\Xi_\DD$ associated with an arbitrary elemental complex $\DD$, and we prove the following two theorems.

Let $G$ be a group acting on a cell complex $\Sigma$.
Recall that the action is said to be \emph{cocompact} if the quotient $\Sigma / G$ is compact.
Furthermore, it is said to be \emph{regular} (in the sense of Bredon) if the stabilizer of each cell fixes that cell pointwise.

\begin{thm}\label{thm1_1}
Let $\DD$ be an elemental complex over $\Gamma$.
\begin{itemize}
\item[(1)]
The cube complex $\Xi_\DD$ is locally regular, connected, and simply connected.
\item[(2)]
The group $\VA[\Gamma]$ acts regularly, cocompactly, and by isometries on $\Xi_\DD$.
\item[(3)]
The stabilizer of each cell is a parabolic subgroup of $\VA [\Gamma]$.
\end{itemize}
\end{thm}

The definition of a parabolic subgroup of $\VA [\Gamma]$ will be recalled in Section~\ref{sec2}.

\begin{thm}\label{thm1_2}
Let $\DD$ be an elemental complex over $\Gamma$.
Then the cube complex $\Xi_\DD$ is CAT(0) if and only if $\DD$ is a flag complex.
\end{thm}

The space $\Xi_\DD$, regarded as a simplicial complex, is independently introduced by Gavazzi and Martin in \cite{GM26} in the case where $\DD = \DD_{\sph}$.
In that paper, the authors endow $\Xi_{\DD_\sph}$ with the so-called \emph{Moussong metric}, which differs from the metric induced by our cube complex structure, and refer to $\Xi_{\DD_\sph}$ as the \emph{virtual Deligne complex}.
The two papers overlap to some extent.
In particular, part of Theorem~\ref{thm1_1} is also proved in \cite{GM26} using essentially the same arguments.
Nevertheless, the two papers address different aspects of the subject. 
In \cite{GM26}, the authors prove that $\Xi_{\DD_\sph}$ endowed with the Moussong metric is CAT(0) whenever $\Gamma$ is locally reducible, whereas in the present paper we prove  that $\Xi_{\DD_\sph}$, endowed with its cube complex structure, is CAT(0) whenever $\Gamma$ is of FC type.
These correspond to the two families of Artin groups studied in \cite{CD95, Cha00}.

The paper is organized as follows.
Let $\DD$ be an elemental complex over $\Gamma$.
In Section~\ref{sec2}, we construct $\Xi_\DD$ as a geometric simplicial complex.
In Section~\ref{sec3}, we endow $\Xi_\DD$ with a cube complex structure.
Theorems~\ref{thm1_1} and~\ref{thm1_2} are proved in Sections~\ref{sec4} and~\ref{sec5}, respectively.

\begin{acknow}
The authors would like to thank Mar\'ia Cumplido, Federica Gavazzi, Anne Lonjou and Alexandre Martin for several stimulating and constructive exchanges.
They also thank the referee for their comments on the manuscript.
\end{acknow}


\section{Virtual Godelle--Paris complex: simplicial complex version}\label{sec2}

In this section, we define the virtual Godelle--Paris complex associated with an elemental complex $\DD$, as a simplicial complex.
In Section~\ref{sec3}, we will show that it admits a natural cube complex structure.

Let $X \subseteq S$.
We denote by $\Gamma_X$ the full Coxeter subgraph of $\Gamma$ spanned by $X$, by $W_X [\Gamma]$ the subgroup of $W [\Gamma]$ generated by $X$, and by $A_X [\Gamma]$ the subgroup of $A [\Gamma]$ generated by $X$.
A subgroup of the form $W_X [\Gamma]$ (resp. $A_X [\Gamma]$) is called a \emph{standard parabolic subgroup} of $W [\Gamma]$ (resp. $A [\Gamma]$), and a subgroup conjugate to $W_X [\Gamma]$ (resp. $A_X [\Gamma]$) is called a \emph{parabolic subgroup} of $W [\Gamma]$ (resp. $A [\Gamma]$).
It is known from \cite{Bou68} that the natural homomorphism $W[\Gamma_X] \to W_X [\Gamma]$ that sends $x$ to $x$ for all $x \in X$ is an isomorphism.
Similarly, it follows from \cite{vLek83} that the natural homomorphism $A [\Gamma_X] \to A_X [\Gamma]$ that sends $x$ to $x$ for all $x \in X$ is an isomorphism.
Henceforth, from now on we identify $W_X [\Gamma]$ with $W[\Gamma_X]$ and $A_X [\Gamma]$ with $A [\Gamma_X]$.

These definitions and results, which are now classical, were extended to virtual Artin groups in \cite{GGP26} as follows.
For $X \subseteq S$, we set $\SS_X = \{ \sigma_x \mid x \in X \}$ and $\TT_X = \{\tau_x \mid x \in X\}$, and we denote by $\VA_X [\Gamma]$ the subgroup of $\VA [\Gamma]$ generated by $\SS_X \sqcup \TT_X$.
Such a subgroup is called a \emph{standard parabolic subgroup}, and a subgroup conjugate to $\VA_X [\Gamma]$ is called a \emph{parabolic subgroup}.
As in the cases of Coxeter and Artin groups, the natural homomorphism $\VA[\Gamma_X] \to \VA_X [\Gamma]$ that sends $\sigma_x$ to $\sigma_x$ and $\tau_x$ to $\tau_x$ for all $x \in X$ is an isomorphism (see \cite[Theorem 1.1]{GGP26}).
Henceforth, from now on we identify $\VA_X [\Gamma]$ with $\VA [\Gamma_X]$.

Let $(P, \le)$ be a partially ordered set.
Recall that the \emph{derived complex} of $(P, \le)$, denoted by $(P, \le)'$, is the abstract simplicial complex whose vertex set is $P$ and whose simplices are the finite chains $\{ x_0 < x_1 < \cdots < x_p\}$.
As ever, the emptyset is considered as a chain.
We denote by $\Delta(P, \le)$ the geometric realization of $(P, \le)'$.

\begin{defin}
Let $\DD$ be an elemental complex over $\Gamma$.
Consider the set
\[
\CC_\DD = \{ g \VA [\Gamma_X] \mid g \in \VA [\Gamma],\ X \in \DD\}\,,
\]
ordered by inclusion.
The \emph{virtual Godelle--Paris complex} associated with $\DD$, denoted by $\Xi_\DD$, is the geometric realization of the derived complex of $\CC_\DD$.
In other words, $\Xi_\DD = \Delta (\CC_\DD, \subseteq)$.
\end{defin}


\section{Virtual Godelle--Paris complex: cube complex version}\label{sec3}

Recall that a \emph{cube complex} is a polyhedral cell complex $E$ such that each cell is isometric to a standard cube $[0,1]^n$ of a Euclidean space, and the gluing maps are isometries.
If the dimensions of the cubes are uniformly bounded, then such a space is a complete geodesic metric space (see \cite{Bri91}).

\begin{rem}
In this paper, we adopt the convention that every cube complex is \emph{regular}, meaning that no two distinct faces of the same cube are identified.
\end{rem}

In this section, we fix an elemental complex $\DD$ over $\Gamma$.
Our goal is to show that $\Xi_\DD$ admits a natural cube complex structure.

Let $X = \{x_1, \dots, x_n \}$ be a finite set of cardinality $n$.
There is a bijection $\zeta_X : \PP (X) \to V([0,1]^n)$ from the power set of $X$ to the vertex set of the standard cube $[0,1]^n$, defined as follows.
for $Z \subseteq X$, we set $\zeta_X (Z) = (\varepsilon_1, \dots, \varepsilon_n)$, where $\varepsilon_i = 1$ if $x_i \in Z$, and $\varepsilon_i=0$ if $x_i \not \in Z$.
This map induces a continuous map $\zeta_X : \Delta (\PP (X), \subseteq) \to [0,1]^n$ that sends each simplex $\Delta(Z_0, Z_1, \dots, Z_p)$ corresponding to a chain $\{ Z_0 \subset Z_1 \subset \cdots \subset Z_p\}$ to the convex hull of $\{ \zeta_X(Z_0), \zeta_X (Z_1), \dots, \zeta_X (Z_p)\}$.
It is well known that this map is a homeomorphism (see, for example, \cite[Lemma A.4.9]{Dav08}), hence we can identify $\Delta (\PP (X), \subseteq)$ with $[0,1]^n$ via $\zeta_X$.
We denote by $\Cub (X)$ the complex $\Delta (\PP (X), \subseteq)$ endowed with the metric induced by this identification.

For $Y, Y' \subseteq X$ such that $Y \subseteq Y'$, we denote by $\Fac(Y, Y', X)$ the convex hull in $\Cub(X)$ of $\{ \zeta_X (Z) \mid Y \subseteq Z \subseteq Y'\}$.
It is easily seen that $\Fac (Y, Y', X)$ is a face of $\Cub (X)$, and that every face of $\Cub (X)$ arises in this way.
Moreover, $\Cub (Y' \setminus Y)$ is isometric to $\Fac (Y, Y' , X)$ via the isometry that sends $\zeta_{Y' \setminus Y} (T)$ to $\zeta_X (Y \cup T)$ for every $T \in \PP (Y' \setminus Y)$.

It follows from the presentations of $\VA [\Gamma]$ and $W [\Gamma]$, that there exists a surjective homomorphism $\pi_K : \VA [\Gamma] \to W [\Gamma]$ that sends $\sigma_s$ to $1$ and $\tau_s$ to $s$ for every $s \in S$.
The kernel of $\pi_K$ is called the \emph{kure virtual Artin group} of $\Gamma$ and is denoted by $\KVA [\Gamma]$.
Moreover, $\pi_K$ admits a section $\iota_K : W [\Gamma] \to \VA [\Gamma]$ that sends $s$ to $\tau_s$ for all $s \in S$.
Consequently, we have the semidirect product decomposition $\VA [\Gamma] \simeq \KVA [\Gamma] \rtimes W [\Gamma]$.
These homomorphisms, together with this semidirect product decomposition, play an important role in the study of virtual Artin groups and, in particular, in the present paper. 

\begin{lem}\label{lem3_1}
Let $X, Y \subseteq S$ and $g,h \in \VA [\Gamma]$.
Then $g \VA [\Gamma_X] \subseteq h \VA [\Gamma_Y]$ if and only if $X \subseteq Y$ and $g^{-1} h \in \VA [\Gamma_Y]$.
In particular, $g \VA [\Gamma_X] = h \VA [\Gamma_Y]$ if and only if $X = Y$ and $g^{-1} h \in \VA [\Gamma_X]$.
\end{lem}

\begin{proof}
Clearly, if $X \subseteq Y$ and $g^{-1} h \in \VA [\Gamma_Y]$, then $g \VA [\Gamma_X] \subseteq g \VA [\Gamma_Y] = h \VA [\Gamma_Y]$.
We now prove the converse.
Assume that $g \VA [\Gamma_X] \subseteq h \VA [\Gamma_Y]$.
Let $u = \pi_K (g)$ and $v = \pi_K (h)$.
Applying $\pi_K$ to the inclusion $g \VA [\Gamma_X] \subseteq h \VA [\Gamma_Y]$ yields $u W [\Gamma_X] \subseteq v W [\Gamma_Y]$.
It is well known that the latter inclusion implies $X \subseteq Y$ (see, for example, \cite[Theorem 4.1.6]{Dav08}).
Finally, $g \in g \VA [\Gamma_X] \subseteq h \VA [\Gamma_Y]$, hence $g^{-1} h \in \VA [\Gamma_Y]$.
\end{proof}

The following result is proved in \cite[Theorem 1.2]{GGP26}.

\begin{lem}[Gálvez-Mateos--Gavazzi--Paris \cite{GGP26}]\label{lem3_2}
Let $X, Y \subseteq S$.
Then
\[
\VA [\Gamma_X] \cap \VA [\Gamma_Y] = \VA [\Gamma_{X \cap Y}]\,.
\]
\end{lem}

For $g \in \VA [\Gamma]$ and $X \in \DD$, we denote by $v(g, X)$ the vertex of $\Xi_\DD$ (viewed as a geometric simplicial complex) corresponding to the chain $\{ g \VA [\Gamma_X] \}$ of length $1$ in $\CC_\DD$.
For $g \in \VA [\Gamma]$ and $X,Y \in \DD$ with $X \subseteq Y$, we denote by $C(g, X, Y)$ the full simplicial subcomplex of $\Xi_\DD$ spanned by $\{ v(g, Z) \mid X \subseteq Z \subseteq Y\}$.
By Lemma \ref{lem3_1}, there exists an isomorphism of simplicial complexes $\xi_{g, X, Y} : \Cub (Y \setminus X) \to C(g, X, Y)$ that sends $\zeta_{Y \setminus X} (T)$ to $v(g, X \cup T)$ for every $T \in \PP (Y \setminus X)$. 
So, $C(g, X, Y)$ can be identified with the standard cube of dimension $|Y \setminus X|$ via this isomorphism.

\begin{lem}\label{lem3_3}
Let $g, h \in \VA [\Gamma]$ and $X, Y, X', Y' \in \DD$ be such that $X \subseteq Y$ and $X' \subseteq Y'$.
\begin{itemize}
\item[(1)]
We have $C (g, X, Y) \cap C (h, X', Y') \neq \emptyset$ if and only if $X \cup X' \subseteq Y \cap Y'$ and $g^{-1} h \in \VA [\Gamma_{Y \cap Y'}]$.
\item[(2)]
Suppose that $C (g, X, Y) \cap C (h, X', Y') \neq \emptyset$.
Then
\[
C (g, X, Y) \cap C (h, X', Y') = C(g, R, Y \cap Y') = C (h, R, Y \cap Y')\,,
\]
where $R$ is the smallest element of $\DD$ satisfying $X \cup X' \subseteq R \subseteq Y \cap Y'$ and $g^{-1} h \in \VA [\Gamma_R]$.
\end{itemize}
\end{lem}

\begin{proof}
{\it Part~(1).}
Suppose that $C (g, X, Y) \cap C (h, X', Y') \neq \emptyset$.
The cubes $C(g, X, Y)$ and $C(h, X', Y')$ are both full simplicial subcomplexes of $\Xi_{\DD}$, hence $C(g, X, Y) \cap C(h, X', Y')$ is also a full simplicial subcomplex of $\Xi_{\DD}$. 
In particular, $C(g, X, Y) \cap C(h, X', Y')$ contains at least one vertex, denoted by $\tilde{v}$.
There exist $Z , Z'\in \DD$ such that $X \subseteq Z \subseteq Y$, $X' \subseteq Z' \subseteq Y'$ and $\tilde v = v(g, Z) = v (h, Z')$.
The latter equality means that $g \VA [\Gamma_Z] = h \VA [\Gamma_{Z'}]$.
Hence, by Lemma~\ref{lem3_1}, we have $Z=Z'$ and $g^{-1} h \in \VA [\Gamma_Z]$.
Since $X \subseteq Z = Z' \subseteq Y $ and $X' \subseteq Z' = Z \subseteq Y'$, we have $X \cup X' \subseteq Z = Z' \subseteq Y \cap Y'$.
Furthermore, since $Z = Z' \subseteq Y \cap Y'$ and $g^{-1} h \in \VA [\Gamma_Z]$, we have $g^{-1} h \in \VA [\Gamma_{Y \cap Y'}]$.

Conversely, suppose that $X \cup X' \subseteq Y \cap Y'$ and $g^{-1} h \in \VA [\Gamma_{Y \cap Y'}]$. 
Let $\tilde v = v(g, Y \cap Y')$. 
Since $X \subseteq X \cup X' \subseteq Y \cap Y' \subseteq Y$, we have $\tilde v \in C (g , X, Y)$. 
Moreover, as $g^{-1} h \in \VA [\Gamma_{Y \cap Y'}]$, we have $\tilde v = v (h, Y \cap Y')$. 
Again, $X' \subseteq X \cup X' \subseteq Y \cap Y' \subseteq Y'$, hence $\tilde v \in C (h, X', Y')$. 
This shows that $C (g, X, Y) \cap C (h, X', Y') \neq \emptyset$.

{\it Part~(2).}
Suppose that $C (g, X, Y) \cap C (h, X', Y') \neq \emptyset$.
By Part~(1), we have $X \cup X' \subseteq Y \cap Y'$ and $g^{-1} h \in \VA [\Gamma_{Y \cap Y'}]$.
Let $\UU = \{ R' \in \DD \mid X \cup X' \subseteq R' \subseteq Y \cap Y' \text{ and } g^{-1} h \in \VA [\Gamma_{R'}] \}$.
We have $Y \cap Y' \in \UU$, hence $\UU \neq \emptyset$.
Moreover, by Lemma~\ref{lem3_2}, if $R_1, R_2 \in \UU$, then $R_1 \cap R_2 \in \UU$.
It follows that $\UU$ has a least element with respect to inclusion, which we denote by $R$.
Note that, since $g^{-1} h \in \VA [\Gamma_R]$, we have $C (g, R, Y \cap Y') = C (h, R, Y \cap Y')$.
It remains to show that $C (g, X , Y) \cap C (h, X', Y') = C (g, R, Y \cap Y')$.
Since these two sets are full simplicial subcomplexes, it suffices to show that every vertex of $C(g, R, Y \cap Y')$ is a vertex of $C(g, X, Y) \cap C(h, X', Y')$, and that every vertex of $C(g, X, Y) \cap C(h, X', Y')$ is a vertex of $C(g, R, Y \cap Y')$.

Let $\tilde v$ be a vertex of $C (g, R, Y \cap Y')$.
There exists $Z \in \DD$ such that $R \subseteq Z \subseteq Y \cap Y'$ and $\tilde v = v(g, Z) = v (h, Z)$.
Since $X \subseteq R \subseteq Z \subseteq Y \cap Y' \subseteq Y$, we have $\tilde v \in C (g, X ,Y)$.
Similarly, $\tilde v \in C (h, X', Y')$.
So, $\tilde v$ is a vertex of $C(g, X, Y) \cap C(h, X', Y')$.

Let $\tilde v$ be a vertex of $C (g, X, Y) \cap C (h, X', Y')$.
As shown in the proof of Part (1), there exists $Z \in \DD$ such that $X \cup X' \subseteq Z \subseteq Y \cap Y'$, $g^{-1} h \in \VA [\Gamma_Z]$, and $\tilde v = v (g, Z) = v(h, Z)$.
Since $Z \in \UU$, by minimality of $R$, we have $R \subseteq Z \subseteq Y \cap Y'$, hence $\tilde v \in C (g, R, Y \cap Y')$.
\end{proof}

The following proposition completes the description of the cube complex structure on $\Xi_\DD$.

\begin{prop}\label{prop3_4}
The collection of cubes $C (g, X, Y)$, where $g \in \VA [\Gamma]$, $X, Y \in \DD$ and $X \subseteq Y$, endows $\Xi_\DD$ with the structure of a cube complex.
\end{prop}

\begin{proof}
We begin by showing that every simplex of $\Xi_\DD$ is contained in a cube.
In particular, this will imply that $\Xi_\DD$ is the union of the cubes $C (g, X, Y)$, where $g \in \VA [\Gamma]$, $X, Y \in \DD$ and $X \subseteq Y$.
Let $\vartheta$ be a simplex of $\Xi_\DD$.
Let $\{ v (g_0, X_0), v (g_1, X_1), \dots, v(g_p, X_p) \}$ be its set of vertices.
After reordering the vertices of $\vartheta$ if necessary, we may assume that
\[
g_0 \VA [\Gamma_{X_0}] \subseteq g_1 \VA [\Gamma_{X_1}] \subseteq \cdots \subseteq g_p \VA [\Gamma_{X_p}]\,.
\]
By Lemma~\ref{lem3_1}, it follows that $X_0 \subseteq X_1 \subseteq \cdots \subseteq X_p$ and that $g_0^{-1} g_i \in \VA [\Gamma_{X_i}]$ for every $0 \le i \le p$.
The latter inclusion implies that $g_0 \VA [\Gamma_{X_i}] = g_i \VA [\Gamma_{X_i}]$, and hence $v(g_0, X_i) = v(g_i, X_i)$ for every $0 \le i \le p$.
It follows that $\vartheta \subseteq C(g_0, X_0, X_p)$.

The proof is completed by observing that, by Lemma~\ref{lem3_3}, the intersection of any two distinct cubes is either empty or a common face of both cubes.
\end{proof}


\section{Proof of Theorem \ref{thm1_1}}\label{sec4}

Let $E$ be a cube complex, and let $v$ be a vertex of $E$.
The \emph{link} of $v$ in $E$ is $\Link_v (E) = \{ x \in E \mid d(v,x) = \varepsilon \}$, where $0< \varepsilon < \frac{1}{2}$.
The cell decomposition of $E$ induces a cell decomposition of $\Link_v (E)$ in which every cell is naturally a simplex.
However, this decomposition is not necessarily a triangulation, in the sense that it is not necessarily a geometric simplicial complex.
For example, if two squares are glued along their boundaries, then the link of any vertex in the resulting square complex is not a simplicial complex.
We say that $E$ is \emph{locally regular} if, for every vertex $v$ of $E$, $\Link_v (E)$ is a simplicial complex.

In this section, we fix an elemental complex $\DD$ over $\Gamma$ and denote by $\Xi_\DD$ the Godelle--Paris cube complex associated with $\DD$, as described in Sections~\ref{sec2} and~\ref{sec3}.
We begin by proving that $\Xi_\DD$ is locally regular.

Let $\tilde v = v(g,X)$ be a vertex of $\Xi_\DD$.
We denote by $\LL_{\tilde v}$ the set of vertices $\tilde u$ of $\Xi_\DD$ that are joined to $\tilde v$ by an edge, that is, for which there exists a $1$-cube in $\Xi_\DD$ with endpoints $\tilde u$ and $\tilde v$.
It is easily verified that a vertex $\tilde u$ of $\Xi_\DD$ belongs to $\LL_{\tilde v}$ if and only if it is of one of the following two forms:
\begin{itemize}
\item
$\tilde u = v(h, X \setminus \{x\})$, where $x \in X$ and $g^{-1} h \in \VA [\Gamma_X]$,
\item
$\tilde u = v(g, X \cup \{ y\})$, where $y \in S \setminus X$ and $X \cup \{y\} \in \DD$.
\end{itemize}

\begin{lem}\label{lem4_1}
Let $\tilde v = v(g,X)$ be a vertex of $\Xi_\DD$.
Let $\EE = \{ v(h_1, X \setminus \{x_1\}), \dots, v(h_p, X \setminus \{x_p\}), v(g, X \cup \{ y_1\}), \dots, v(g, X \cup \{ y_q\}) \}$ be a finite subset of $\LL_{\tilde v}$, where $x_1, \dots, x_p \in X$, $g^{-1} h_i \in \VA [\Gamma_X]$ for every $1 \le i \le p$, $y_1, \dots, y_q \in S \setminus X$, and $X \cup \{ y_j \} \in \DD$ for every $1 \le j \le q$.
Let $C(k, Z, Z')$ be a cube of $\Xi_\DD$ containing $\tilde v$.
If $C(k, Z, Z') \cap \LL_{\tilde v} = \EE$, then $Z =  X \setminus \{x_1, \dots, x_p\}$ and $Z' = X \cup \{y_1, \dots, y_q\}$.
\end{lem}

\begin{proof}
Let $Z_0 = X \setminus \{x_1, \dots, x_p\}$.
By definition, for every $1 \le i \le p$, we have $Z \subseteq X \setminus \{x_i\}$, hence  $Z \subseteq Z_0$.
Since $\tilde v = v(g, X) \in C(k,Z,Z')$, we have $k^{-1} g \in \VA [\Gamma_X]$, hence $\tilde v = v (k, X)$.
Suppose that $Z \subsetneq Z_0$.
Choose $x' \in Z_0 \setminus Z$.
Then $v(k, X \setminus \{x'\}) \in C (k, Z, Z') \cap \LL_{\tilde v}$, while $v(k, X \setminus \{x'\}) \not \in \EE$.
This is a contradiction, hence $Z = Z_0$.

Let $Z_0' = X \cup \{y_1, \dots, y_q\}$.
By definition, for every $1 \le j \le q$, we have $X \cup \{y_j\} \subseteq Z'$, hence $Z_0' \subseteq Z'$.
In particular, since $Z' \in \DD$, we have $Z_0' \in \DD$.
Suppose that $Z_0' \subsetneq Z'$.
Choose $y' \in Z' \setminus Z_0'$.
Since $k^{-1} g \in \VA [\Gamma_X] \subseteq \VA [\Gamma_{X \cup \{ y' \}}]$, we have $v(k, X \cup \{y'\}) = v(g, X \cup \{ y'\})$, hence $v(k, X \cup \{y'\}) \in C(k, Z, Z') \cap \LL_{\tilde v}$, while $v(k, X \cup \{y'\}) \not \in \EE$.
This is a contradiction, hence $Z' = Z_0'$.
\end{proof}

\begin{lem}\label{lem4_2}
The cube complex $\Xi_\DD$ is locally regular.
\end{lem}

\begin{proof}
Let $\tilde v = v(g,X)$ be a vertex of $\Xi_\DD$.
Consider a finite subset $\EE = \{ v(h_1, X \setminus \{x_1\}), \dots, 
v(h_p, X \setminus \{x_p\}), v(g, X \cup \{ y_1\}), \dots, v(g, X \cup \{ y_q\}) \}$ of $\LL_{\tilde v}$.
We claim that there exists at most one cube $C(k,Z,Z')$ of $\Xi_\DD$ such that $\tilde v \in C (k,Z,Z')$ and $C(k, Z, Z') \cap \LL_{\tilde v} = \EE$.
Suppose, for a contradiction, that there exist two distinct cubes $C(k_1,Z_1,Z_1')$ and $C(k_2, Z_2, Z_2')$ such that $\tilde v \in C (k_1, Z_1, Z_1')$, $\tilde v \in C (k_2, Z_2, Z_2')$, and $C (k_1, Z_1, Z_1') \cap \LL_{\tilde v} = C (k_2, Z_2, Z_2') \cap \LL_{\tilde v} = \EE$.
By Lemma~\ref{lem4_1}, $Z_1 = Z_2 = X \setminus \{x_1, \dots, x_p\}$ and $Z_1' = Z_2' = X \cup \{y_1, \dots, y_q\}$. 
For every $1 \le i \le p$, we have $k_1^{-1} h_i \in \VA [\Gamma_{X \setminus \{ x_i \}}]$ and $k_2^{-1} h_i \in \VA [\Gamma_{X \setminus \{ x_i \}}]$, hence $k_1^{-1} k_2 \in \VA [\Gamma_{X \setminus \{ x_i \}}]$.
Applying Lemma~\ref{lem3_2}, we obtain $k_1^{-1} k_2 \in \VA [\Gamma_{X \setminus \{ x_1, \dots, x_p \}}] = \VA [\Gamma_{Z_1}]$.
It follows that $C (k_1, Z_1, Z_1') = C (k_2, Z_2, Z_2')$, which is a contradiction since we assumed that $C (k_1, Z_1, Z_1') \neq C (k_2, Z_2, Z_2')$.

Let $\tilde u \in \LL_{\tilde v}$.
By definition, there exists a $1$-cube of $\Xi_\DD$ with endpoints $\tilde u$ and $\tilde v$, and by the previous discussion, this $1$-cube is unique. 
Hence, the $0$-skeleton of $\Link_{\tilde v}(\Xi_\DD)$ is in bijection with $\LL_{\tilde v}$.
So, without loss of generality, we may identify $\LL_{\tilde v}$ with the $0$-skeleton of $\Link_{\tilde v}(\Xi_\DD)$.
Let $D$ be a cell of $\Link_{\tilde v}(\Xi_\DD)$.
By construction, $D$ is a simplex, and by the previous discussion, $D$ is completely determined by its set of vertices.
This shows that $\Link_{\tilde v}(\Xi_\DD)$ is a geometric simplicial complex.
\end{proof}

The proof of Theorem~\ref{thm1_1} relies in part on the theory of simple complexes of groups.
We present only the aspects of the theory that are needed for the proof of Theorem~\ref{thm1_1} and refer to \cite[Chapter II.12]{BH99} for a detailed exposition of the subject.

Let $(\PP, \preceq)$ be a partially ordered set such that $\Delta (\PP, \preceq)$ is simply connected.
A \emph{simple complex of groups} $\GG$ over $\PP$ consists of the following data: \begin{itemize}
\item
for each $x \in \PP$, a group $G_x$;
\item
for every pair $x,y \in \PP$ with $x \prec y$, an injective homomorphism $\varphi_{x,y} : G_x \to G_y$.
\end{itemize}
These data must satisfy the following condition:
\begin{itemize}
\item
for every $x,y,z \in \PP$ with $x \prec y \prec z$, $\varphi_{x,z} = \varphi_{y,z} \circ \varphi_{x,y}$.
\end{itemize}

\begin{rem}
In the definition of a simple complex of groups given in \cite{BH99}, there is an injective homomorphism $\varphi_{x,y} : G_x \to G_y$ whenever $x \succ y$, rather than when $x \prec y$, as here.
In practice, this difference in convention does not matter, since we can always replace $\PP$ with its opposite poset.
\end{rem}

The \emph{fundamental group} of the complex of groups $\GG$, denoted by $\pi_1(\GG)$, is the quotient $\left(\ast_{x \in \PP} G_x\right)/\sim$, where $\sim$ is the equivalence relation generated by
\[
\forall x,y \in \PP \text{ with } x \prec y,\ \forall g \in G_x,\ g \sim \varphi_{x,y}(g)\,.
\]
In other words, $\pi_1 (\GG) = \varinjlim_{x \in \PP} G_x$ is the direct limit of the system of groups and monomorphisms $\{G_x, \varphi_{x,y}\}$.
For each $x \in \PP$, we denote by $\psi_x : G_x \to \pi_1 (\GG)$ the homomorphism induced by the inclusion of $G_x$ into $\ast_{y \in \PP} G_y$.
We say that $\GG$ is \emph{developable} if $\psi_x$ is injective for every $x \in \PP$.

\begin{rem}
Non-developable simple complexes of groups are not easy to construct.
Examples can be found in \cite[Examples~II.12.17]{BH99}.
\end{rem}

Suppose that $\GG$ is a simple complex of groups over a partially ordered set $(\PP, \preceq)$ as above.
Suppose that $\GG$ is developable, and set $G = \pi_1 (\GG)$.
For each $x \in \PP$, as the homomorphism $\psi_x : G_x \to G$ is injective, we can and do identify $G_x$ with its image under $\psi_x$.
Let $\approx$ be the equivalence relation on $G \times \PP$ defined by
\[
(g_1,x_1) \approx (g_2,x_2) \text{ if } x_1 = x_2 \text{ and } g_1^{-1}g_2 \in G_{x_1}\,.
\]
Set $\hat \PP (\GG) = (G \times \PP)/\approx$.
The equivalence class of a pair $(g,x)$ is denoted by $[g,x]$.
Define a relation $\preceq$ on $\hat \PP(\GG)$ by
\[
[g,x] \preceq [h,y] \text{ if } x \preceq y \text{ and } g^{-1} h \in G_y\,.
\]
It is easy to verify that this relation is well defined (that is, it does not depend on the choice of $g$ and $h$), and is a partial order on $\hat \PP (\GG)$.
The \emph{universal cover} of $\GG$ is then defined by $U(\GG) = \Delta(\hat \PP (\GG), \preceq)$.
Note that $G$ acts on $\hat \PP (\GG)$ by $g \cdot [h,x] = [gh,x]$, and this action induces an action of $G$ on $U(\GG)$.

The following theorem can be found in \cite[Part~II, Corollary~12.21]{BH99}.
The original version is due to Haefliger \cite{Hae91}.

\begin{thm}[Haefliger \cite{Hae91}]\label{thm4_3}
Let $\GG$ be a simple complex of groups over a partially ordered set $(\PP, \preceq)$ as above.
Assume that $\GG$ is developable, and set $G = \pi_1 (\GG)$.
\begin{itemize}
\item [(1)]
$U(\GG)$ is connected and simply connected.
\item[(2)]
$U(\GG)/G = \Delta(\PP, \preceq)$.
\end{itemize}
\end{thm}

Recall that we are given an elemental complex $\DD$ ordered by inclusion.
For each $X \in \DD$, let $v(X)$ denote the vertex of $\Delta (\DD, \subseteq)$ corresponding to the chain $\{ X \}$ of length $1$.
Let $\Delta^* (\DD, \subseteq)$ be the full geometric subcomplex of $\Delta (\DD, \subseteq)$ spanned by $\{ v(X) \mid X \in \DD \setminus \{ \emptyset\}\}$.
It it easy to see that $\Delta (\DD, \subseteq)$ is the cone over $\Delta^* (\DD, \subseteq)$ with apex $v(\emptyset)$, hence $\Delta(\DD, \subseteq)$ is simply connected.
Now, consider the simple complex of groups $\GG = \GG_\DD [\Gamma]$ over $(\DD, \subseteq)$, where $G_X = \VA [\Gamma_X]$ for every $X \in \DD$ and where, for $X \subset Y$, $\varphi_{X,Y} : \VA [\Gamma_X] \to \VA [\Gamma_Y]$ is the canonical homomorphism.
Recall that, by \cite[Theorem~1.1]{GGP26}, each such homomorphism is injective.

\begin{rem}
The assumption that $\{x,y\} \in \DD$ for all distinct $x,y \in S$ such that $m_{x,y} \neq \infty$ is a key assumption for the following lemma.
\end{rem}

\begin{lem}\label{lem4_4}
Let $\GG = \GG_\DD [\Gamma]$ be as above.
Then $\pi_1 (\GG) = \VA [\Gamma]$ and, for every $X \in \DD$, the homomorphism $\psi_X : \VA [\Gamma_X] \to \pi_1 (\GG) = \VA [\Gamma]$ is the canonical homomorphism.
In particular, $\GG_\DD [\Gamma]$ is developable.
\end{lem}

\begin{proof}
For $X \in \DD$, let $\iota_X : \VA [\Gamma_X] \to \VA [\Gamma]$ denote the canonical homomorphism.
The family $\{ \iota_X : \VA [\Gamma_X] \to \VA [\Gamma] \mid X \in \DD\}$ induces a homomorphism $\hat \Phi : \ast_{X \in \DD} \VA [\Gamma_X] \to \VA [\Gamma]$.
Observe that, for $X,Y \in \DD$ with $X \subset Y$ and $g \in \VA [\Gamma_X]$, we have $\hat \Phi (g) = \hat \Phi ( \varphi_{X,Y} (g))$.
Hence $\hat \Phi$ induces a homomorphism $\Phi : \pi_1 (\GG) \to \VA [\Gamma]$.

Let $F( \SS \cup \TT)$ denote the free group freely generated by $\SS \cup \TT$, and let $\hat \Psi : F (\SS \cup \TT) \to \pi_1 (\GG)$ be the homomorphism defined by $\hat \Psi (\sigma_s) = \psi_{\{ s \}} (\sigma_s)$ and $\hat \Psi (\tau_s) = \psi_{\{ s \}} (\tau_s)$ for all $s \in S$. 
Let $s \in S$. 
Then
\[
\hat \Psi (\tau_s)^2 =
\psi_{\{ s \}} (\tau_s)^2 =
\psi_{\{ s \}} (\tau_s^2) =
\psi_{\{ s \}} (1) = 1\,.
\]
Let $s,t \in S$ distinct with $m_{s,t} \neq \infty$. 
Note that, by definition, $\{s,t\} \in \DD$. 
Then
\begin{gather*}
(\hat \Psi (\sigma_s), \hat \Psi (\sigma_t))_{m_{s,t}} =
( \psi_{\{ s \}} (\sigma_s),  \psi_{\{ t \}} (\sigma_t))_{m_{s,t}} =
( \psi_{\{ s,t \}} \circ \varphi_{\{s\}, \{s,t\}} (\sigma_s), \psi_{\{ s,t \}} \circ \varphi_{\{t\}, \{s,t\}} (\sigma_t))_{m_{s,t}} = \\
\psi_{\{ s,t \}} \big(( \sigma_s, \sigma_t)_{m_{s,t}} \big) =
\psi_{\{ s,t \}} \big(( \sigma_t, \sigma_s)_{m_{s,t}} \big) =
( \psi_{\{ s,t \}} \circ \varphi_{\{t\}, \{s,t\}} (\sigma_t), \psi_{\{ s,t \}} \circ \varphi_{\{s\}, \{s,t\}} (\sigma_s))_{m_{s,t}} =\\
( \psi_{\{ t \}} (\sigma_t),  \psi_{\{ s \}} (\sigma_s))_{m_{s,t}} =
(\hat \Psi (\sigma_t), \hat \Psi (\sigma_s))_{m_{s,t}}\,.
\end{gather*}
We prove in the same way that
\begin{gather*}
(\hat \Psi (\tau_s), \hat \Psi (\tau_t))_{m_{s,t}} =
(\hat \Psi (\tau_t), \hat \Psi (\tau_s))_{m_{s,t}}\,,\\
(\hat \Psi (\tau_s), \hat \Psi (\tau_t))_{m_{s,t}-1} \hat \Psi (\sigma_s) =
\hat \Psi (\sigma_x) (\hat \Psi (\tau_s), \hat \Psi (\tau_t))_{m_{s,t}-1}\,,
\end{gather*}
where $x =s$ if $m_{s,t}$ is even, and $x=t$ if $m_{s,t}$ is odd.  
It follows that $\hat \Psi$ induces a homomorphism $\Psi : \VA [\Gamma] \to \pi_1 (\GG)$.

It is easy to verify that $\Psi \circ \Phi = \id_{\pi_1 (\GG)}$ and $\Phi \circ \Psi = \id_{\VA [\Gamma]}$, hence $\Phi$ and $\Psi$ are isomorphisms.
Finally, by \cite[Theorem 1.1]{GGP26}, for every $X \in \DD$, the homomorphism $\psi_X = \iota_X : \VA [\Gamma_X] \to \VA [\Gamma]$ is injective, hence the complex of groups $\GG_\DD [\Gamma]$ is developable.
\end{proof}

\begin{proof}[Proof of Theorem \ref{thm1_1}]
There is a natural action of $\VA [\Gamma]$ on $\Xi_\DD$, defined on the vertices by $g \cdot v (h, X) = v (gh, X)$.
Note that, if $g \in \VA [\Gamma]$ and $C (h, X, Y)$ is a cube of $\Xi_\DD$, then $g \cdot C (h, X, Y) = C (gh, X,Y)$.
This is clearly an action by isometries.
Furthermore, by Lemma~\ref{lem4_2}, the cube complex $\Xi_\DD$ is locally regular.

By Lemma~\ref{lem3_1}, there is an isomorphism of posets $\hat \DD (\GG_\DD [\Gamma]) \to \CC_\DD = \{ g \VA[\Gamma_X] \mid g \in \VA [\Gamma],\ X \in \DD\}$ which sends a class $[g,X]$ to the coset $g \VA [\Gamma_X]$.
This isomorphism induces an isomorphism of geometric simplicial complexes $U (\GG_\DD [\Gamma]) \to \Xi_\DD = \Delta (\CC_\DD, \subseteq)$ that is equivariant under the action of $\VA [\Gamma]$.
By Theorem~\ref{thm4_3}, it follows that $\Xi_\DD$ is connected and simply connected.
Moreover, since $\DD$ is finite, $\Xi_\DD/\VA[\Gamma] \simeq \Delta (\DD, \subseteq)$ is compact, hence the action of $\VA [\Gamma]$ on $\Xi_\DD$ is cocompact.

Let $g \in \VA [\Gamma]$, and let $C (h, X, Y)$ be a cube of $\Xi_\DD$. 
Applying Lemma ~\ref{lem3_1} we obtain: 
\begin{gather*}
g \cdot C (h, X, Y) = C (h, X, Y)\ \Leftrightarrow\\
\forall Z \in \DD \text{ with } X \subseteq Z \subseteq Y,\ g \cdot v(h, Z) = v(gh, Z) = v (h, Z)\ \Leftrightarrow\\
\forall Z \in \DD \text{ with } X \subseteq Z \subseteq Y,\ h^{-1} g h \in \VA [\Gamma_Z]\ \Leftrightarrow\\
\forall Z \in \DD \text{ with } X \subseteq Z \subseteq Y,\ g \in h \VA [\Gamma_Z] h^{-1}\ \Leftrightarrow\\
g \in h \VA [\Gamma_X] h^{-1}\,.
\end{gather*}
Thus, the action of $\VA[\Gamma]$ on $\Xi_\DD$ is regular, and the stabilizer of the cube $C(h,X,Y)$ is the parabolic subgroup $h\VA[\Gamma_X]h^{-1}$.
\end{proof}


\section{Proof of Theorem \ref{thm1_2}}\label{sec5}

We keep the notation of the previous section. 
So, $\DD$ denotes an elemental complex over $\Gamma$, and $\Xi_\DD$ is the virtual Godelle--Paris complex associated with it.

We begin by recalling the definition of a CAT(0) space.
A \emph{geodesic segment} in a metric space $(E, d)$ is an isometric embedding of an interval $[0, l]$ into $E$.
A metric space $(E, d)$ is said to be \emph{geodesic} if it is complete and every pair of points in $E$ is joined by a geodesic segment.
A \emph{geodesic triangle} $T$ consists of three geodesic segments $\gamma_1, \gamma_2, \gamma_3$ joining three points in $E$.
A \emph{comparison triangle} for $T$ in the Euclidean plane $\E^2$ is a geodesic triangle $\bar T = (\bar \gamma_1, \bar \gamma_2, \bar \gamma_3)$ in $\E^2$ such that the length of $\gamma_i$ is equal to that of $\bar \gamma_i$ for each $i \in \{1, 2, 3\}$.
Note that a comparison triangle always exists and is unique up to isometry.
For a point $x$ on $\gamma_i$, we denote by $\bar x$ the corresponding point on $\bar \gamma_i$.
The triangle $T$ is said to satisfy the \emph{CAT(0) inequality} if, for every pair of points $x, y \in T$, the distance between $x$ and $y$ in $E$ is less or equal to the distance between $\bar x$ and $\bar y$ in $\E^2$.
A geodesic metric space $(E, d)$ is called a \emph{CAT(0) space} if every geodesic triangle in $E$ satisfies the CAT(0) inequality.

CAT(0) spaces were introduced by Gromov \cite{Gro87} in his seminal paper on geometric group theory.
Their study, together with that of discrete groups acting on them, constitute cornerstones of the field, the case of CAT(0) cube complexes being particularly important.
We refer to \cite{BH99} for a complete and detailed account of the subject.

Let $E$ be a cube complex, and let $v$ be a vertex of $E$.
The link $\Link_v(E)$ is defined at the beginning of Section~\ref{sec4} as a cell complex whose cells are simplices (such a complex is called a $\Delta$-complex).
If $E$ is locally regular, then $\Link_v(E)$ is a geometric simplicial complex. However, to simplify the proofs and statements that follow, we will in this case identify $\Link_v(E)$ with the abstract simplicial complex that defines it.
In particular, for every vertex $\tilde v$ of $\Xi_\DD$, the link $\Link_{\tilde v}(\Xi_\DD)$ will be regarded as an abstract simplicial complex.

To determine when the cube complex $\Xi_\DD$ is CAT(0), we use the following well-known criterion due to Gromov \cite{Gro87}.
Recall that an abstract simplicial complex $\Sigma$ on a vertex set $V$ is a \emph{flag complex} if, for every finite subset $\AA \subseteq V$, the following equivalence holds: 
\[
\big( \forall x,y \in \AA,\ \{x,y \} \in \Sigma\big)\ \Leftrightarrow\
\AA \in \Sigma\,.
\]

\begin{thm}[Gromov~\cite{Gro87}]\label{thm5_1}
Let $E$ be a connected, simply connected, locally regular, finite-dimensional cube complex.
Then $E$ is CAT(0) if and only if $\Link_v (E)$ is a flag complex for every vertex $v$ of $E$.
\end{thm}

To prove Theorem~\ref{thm1_2}, we follow the same strategy as in \cite[Section 4]{GP12}.
Let $\tilde v = v(g, X)$ be a vertex of $\Xi_\DD$.
We first show that $\Link_{\tilde v} (\Xi_\DD)$ is isomorphic to the join of two simplicial complexes, denoted by $\VArt[\Gamma_X]$ and $\DD^X$, respectively.
We show in Corollary~\ref{corl5_3} that $\VArt[\Gamma_X]$ is always a flag complex.
Thus, $\Link_{\tilde v} (\Xi_\DD)$ is a flag complex if and only if $\DD^X$ is a flag complex.
Then we show that $\DD^X$ is a flag complex for every $X \in \DD$ if and only if $\DD$ itself is a flag complex.
We begin with the definitions of $\DD^X$ and $\VArt [\Gamma]$.

Let $X \in \DD$.
Define $\DD^X$ to be the abstract simplicial complex whose vertex set is $\DD_0^X = \{ y \in S \setminus X \mid X \cup \{ y \} \in \DD \}$, and in which a subset $Y \subseteq \DD_0^X$ is a simplex if and only if $X \cup Y \in \DD$.
Note that $\DD_0^X$ may be empty.
In this case, $\DD^X$ consists of a unique simplex, $\emptyset$.

The \emph{virtual Artin complex} of $\Gamma$, denoted by $\VArt[\Gamma]$, is the abstract simplicial complex defined as follows:
\begin{itemize}
\item
The vertex set of $\VArt [\Gamma]$ is
$
\{ g \VA [\Gamma_{S \setminus \{s\}}] \mid s \in S,\ g \in \VA [\Gamma]\}\,.
$
For convenience, we write $\xi(g,s) = g \VA [\Gamma_{S \setminus \{s\}}]$, which is regarded as a vertex of $\VArt [\Gamma]$.
\item
A finite non-empty set of vertices $\{ \xi (g_0, s_0), \xi (g_1, s_1), \dots, \xi (g_p, s_p)\}$ is a simplex of $\VArt [\Gamma]$ if 
\[
\bigcap_{i=0}^p g_i \VA [\Gamma_{S \setminus \{s_i\}}] \neq \emptyset\,.
\]
\end{itemize}
As usual, the empty set is also regarded as a simplex of $\VArt [\Gamma]$.

The following result has the immediate consequence that $\VArt [\Gamma]$ is a flag complex (see Corollary~\ref{corl5_3}), but it is also of independent interest.

\begin{prop}\label{prop5_2}
Let $g_1, \dots, g_p \in \VA [\Gamma]$ and let $X_1, \dots, X_p \subseteq S$.
Suppose that $g_i \VA [\Gamma_{X_i}] \cap g_j \VA [\Gamma_{X_j}] \neq \emptyset$ for every pair of distinct indices $i,j \in \{1, \dots, p\}$.
Then
\[
\bigcap_{i=1}^p g_i \VA [\Gamma_{X_i}] \neq \emptyset\,.
\]
\end{prop}

\begin{corl}\label{corl5_3}
$\VArt [\Gamma]$ is a flag complex.
\end{corl}

A result analogous to Proposition~\ref{prop5_2}, which will be used in its proof, is known for Coxeter and Artin groups.
The case of Coxeter groups has been known to specialists for a long time and can be found for instance in \cite[Lemma 2.7]{HO21} and \cite[Exercice 3.116]{AB08}.
The case of Artin groups is proved in \cite[Lemma 4.7]{GP12}.

\begin{prop}\label{prop5_4}\leavevmode
\begin{itemize}
\item[(1)]
Let $w_1, \dots, w_p \in W [\Gamma]$ and let $X_1, \dots, X_p \subseteq S$.
Suppose that $w_i W [\Gamma_{X_i}] \cap w_j W [\Gamma_{X_j}] \neq \emptyset$ for every pair of distinct indices $i,j \in \{1, \dots, p\}$.
Then $\bigcap_{i=1}^p w_i W [\Gamma_{X_i}] \neq \emptyset$.
\item[(2)]
Let $g_1, \dots, g_p \in A [\Gamma]$ and let $X_1, \dots, X_p \subseteq S$.
Suppose that $g_i A [\Gamma_{X_i}] \cap g_j A [\Gamma_{X_j}] \neq \emptyset$ for every pair of distinct indices $i,j \in \{1, \dots, p\}$.
Then $\bigcap_{i=1}^p g_i A [\Gamma_{X_i}] \neq \emptyset$.
\end{itemize}
\end{prop}

Recall the homomorphism $\pi_K : \VA [\Gamma] \to W [\Gamma]$ defined by $\pi_K (\sigma_s) = 1$ and $\pi_K (\tau_s) = s$ for every $s \in S$.
The kernel of $\pi_K$ is the \emph{kure virtual Artin group} of $\Gamma$, denoted by $\KVA [\Gamma]$.
The homomorphism $\pi_K$ admits a section $\iota_K : W [\Gamma] \to \VA [\Gamma]$, which sends $s$ to $\tau_s$ for every $s \in S$, and we have the semidirect product decomposition $\VA [\Gamma] \simeq \KVA [\Gamma] \rtimes W [\Gamma]$.
The proof of Proposition~\ref{prop5_2} relies heavily on the study of $\KVA [\Gamma]$ carried out in \cite{BPT23,GGP26}.
So, we start by recalling some preliminary results on $\KVA [\Gamma]$ needed for this proof.

Let $\Pi = \{\alpha_s \mid s \in S\}$ be a set in bijection with $S$.
Let $V = \bigoplus_{s \in S} \R \alpha_s$ be the real vector space with basis $\Pi$, and let $\langle \cdot, \cdot \rangle : V \times V \to \R$ be the symmetric bilinear form defined by:
\[
\langle \alpha_s, \alpha_t \rangle = \left\{ \begin{array}{ll}
- 2 \cos (\pi/m_{s,t}) & \text{if } m_{s,t} \neq \infty\,,\\
-2 & \text{if } m_{s,t} = \infty\,.
\end{array} \right.
\]
There exists a faithful linear representation $\rho : W \to \GL (V)$, preserving the bilinear form $\langle \cdot, \cdot \rangle$, defined by:
\[
\rho (s) (v) = v - \langle v, \alpha_s \rangle \alpha_s\,,\quad
\text{for every } v \in V \text{ and } s \in S\,.
\]
For $w \in W [\Gamma]$ and $v \in V$, we will write $w \cdot v$ instead of $\rho(w) (v)$.
The set $\Phi [\Gamma] = \{ w \cdot \alpha_s \mid w \in W[\Gamma],\ s \in S\}$ is called the \emph{root system} of $\Gamma$.

We define a Coxeter matrix $\hat M = (\hat m_{\beta, \gamma})_{\beta, \gamma \in \Phi [\Gamma]}$ on $\Phi [\Gamma]$ as follows.
\begin{itemize}
\item[(a)]
Set $\hat m_{\beta, \beta} = 1$ for every $\beta \in \Phi [\Gamma]$.
\item[(b)]
Let $\beta, \gamma \in \Phi [\Gamma]$ be distinct.
Suppose that there exist $s, t \in S$ and $w \in W [\Gamma]$ such that $w \cdot \alpha_s = \beta$, $w \cdot \alpha_t = \gamma$, and $m_{s,t} \neq \infty$.
Then set $\hat m_{\beta, \gamma} = m_{s,t}$.
\item[(c)]
Set $\hat m_{\beta, \gamma} = \infty$ in the other cases.
\end{itemize}
It is shown in \cite{BPT23} that the definition of $\hat m_{\beta, \gamma}$ does not depend on the choice of $s,t \in S$ and $w \in W [\Gamma]$ in~(b).
The Coxeter graph of $\hat M$ is denoted by $\hat \Gamma$, and the standard generating set of $A [\hat \Gamma]$ is denoted by $\{ \hat \delta_\beta \mid \beta \in \Phi [\Gamma] \}$.

\begin{rem}
It follows from \cite{Deo82} that $\Phi [\Gamma]$ is finite if and only if $W [\Gamma]$ is finite.
In particular, if $W [\Gamma]$ is infinite, then $\hat \Gamma$ is a Coxeter graph with infinitely many vertices.
However, the results on Artin groups with finitely many standard generators that we use extend straightforwardly to Artin groups with infinitely many standard generators.
\end{rem}

Let $\beta \in \Phi [\Gamma]$.
Choose $w \in W [\Gamma]$ and $s \in S$ such that $\beta = w \cdot \alpha_s$, and set $\delta_\beta = \iota_K (w)\, \sigma_s\, \iota_K (w)^{-1} \in \KVA [\Gamma]$.
By \cite[Lemma~2.2]{BPT23}, this definition does not depend on the choice of $w$ and $s$.
The Coxeter graph $\hat \Gamma$ and the group $\KVA [\Gamma]$ are related by the following result, proved in \cite[Theorem~2.3]{BPT23}.

\begin{thm}[Bellingeri--Paris--Thiel \cite{BPT23}]\label{thm5_5}
The map $\{ \hat \delta_\beta \mid \beta \in \Phi [\Gamma]\} \to \{ \delta_\beta \mid \beta \in \Phi [\Gamma]\}$, $\hat \delta_\beta \mapsto \delta_\beta$, induces an isomorphism $\varphi : A [\hat \Gamma] \to \KVA [\Gamma]$.
\end{thm}

From now on, we identify $A [\hat \Gamma]$ with $\KVA [\Gamma]$ via the above isomorphism.

Let $X \subseteq S$.
We identify $V_X=\bigoplus_{x\in X} \R \alpha_x$ with the corresponding vector subspace of $V = \bigoplus_{s \in S} \R \alpha_s$, and, under this identification, we regard
$\Phi[\Gamma_X]$ as a subset of $\Phi[\Gamma]$.
The following result is a direct consequence of \cite[Lemma~3.8]{GGP26}.

\begin{lem}[Gálvez-Mateos--Gavazzi--Paris \cite{GGP26}]\label{lem5_6}
Let $X \subseteq S$.
Then the canonical embedding $\VA[\Gamma_X] \hookrightarrow \VA [\Gamma]$ restricts to an isomorphism from $\KVA [\Gamma_X] = A [\widehat{\Gamma_X}]$ to $A [\hat \Gamma_{\Phi[\Gamma_X]}]$, the subgroup of $\KVA [\Gamma] = A [\hat \Gamma]$ generated by $\{\delta_\beta \mid \beta \in \Phi [\Gamma_X]\}$.
\end{lem}

In the proof of Proposition \ref{prop5_2}, we will also use the following result.

\begin{lem}\label{lem5_6B}
Let $X \subseteq S$.
Then $\VA [\Gamma_X] \cap \KVA [\Gamma] = \KVA [\Gamma_X]$.
\end{lem}

\begin{proof}
Recall that $\KVA [\Gamma]$ is the kernel of the homomorphism $\pi_K : \VA [\Gamma] \to W [\Gamma]$ that maps $\sigma_s$ to $1$ and $\tau_s$ to $s$ for every $s \in S$.
The restriction of $\pi_K$ to $\VA [\Gamma_X]$ is the homomorphism $\pi_K : \VA [\Gamma_X] \to W [\Gamma_X]$, whose kernel is $\KVA [\Gamma_X]$, hence $\VA [\Gamma_X] \cap \KVA [\Gamma] = \KVA [\Gamma_X]$.
\end{proof}

\begin{proof}[Proof of Proposition~\ref{prop5_2}]
Let $w_i = \pi_K (g_i)$ for each $1 \le i \le p$.
Let $i,j \in \{1, \dots, p\}$ be distinct.
Applying $\pi_K$ to the nonempty intersection $g_i \VA [\Gamma_{X_i}] \cap g_j \VA [\Gamma_{X_j}] \neq \emptyset$, we obtain $w_i W [\Gamma_{X_i}] \cap w_j W[\Gamma_{X_j}] \neq \emptyset$.
By Proposition~\ref{prop5_4}\,(1), it follows that $\bigcap_{i=1}^p w_i W [\Gamma_{X_i}] \neq \emptyset$.
Choose $v \in \bigcap_{i=1}^p w_i W [\Gamma_{X_i}]$.
Let $i \in \{1, \dots, p\}$.
Since $v \in w_i W [\Gamma_{X_i}]$, there exists $v_i \in W [\Gamma_{X_i}]$ such that $w_i = v v_i$.
Set $h_i = g_i \, \iota_K (v_i)^{-1}$.
Since $v_i \in W [\Gamma_{X_i}]$, we have $\iota_K (v_i) \in \VA [\Gamma_{X_i}]$, hence $h_i \VA [\Gamma_{X_i}] = g_i \VA [\Gamma_{X_i}]$.
Moreover, $\pi_K (h_i) = w_i v_i^{-1} = v$.

For each $i \in \{1, \dots, p\}$, set $k_i = \iota_K (v)^{-1} \, h_i$.
Since $\pi_K (k_i) = 1$, we have $k_i \in \KVA [\Gamma] = A [\hat \Gamma]$.
Let $i,j \in \{1, \dots, p\}$ be distinct.
Then
\begin{gather*}
k_i \VA [\Gamma_{X_i}] \cap k_j \VA [\Gamma_{X_j}] =
\iota_K (v)^{-1} \big( h_i \VA [\Gamma_{X_i}] \cap h_j \VA [\Gamma_{X_j}] \big) =\\
\iota_K (v)^{-1} \big( g_i \VA [\Gamma_{X_i}] \cap g_j \VA [\Gamma_{X_j}] \big)
\neq \emptyset\,.
\end{gather*}
Choose $a_{i,j} \in k_i \VA [\Gamma_{X_i}] \cap k_j \VA [\Gamma_{X_j}]$.
Set $u_{i,j} = \pi_K (a_{i,j}) \in W [\Gamma_{X_i}] \cap W [\Gamma_{X_j}]$.
It is well known that $W [\Gamma_{X_i}] \cap W [\Gamma_{X_j}] = W [\Gamma_{X_i \cap X_j}]$ (see, for example, \cite[Theorem 4.16]{Dav08}), hence $\iota_K (u_{i,j}) \in \VA [\Gamma_{X_i \cap X_j}] = \VA [\Gamma_{X_i}] \cap \VA [\Gamma_{X_j}]$ (see Lemma \ref{lem3_2}).

Define 
\[
b_{i,j} = a_{i,j} \, \iota_K (u_{i,j})^{-1} \in  k_i \VA [\Gamma_{X_i}] \cap k_j \VA [\Gamma_{X_j}]\,.
\]
We have $\pi_K (b_{i,j}) =1$ and $\pi_K (k_i) =1$, hence $k_i^{-1} b_{i,j} \in \KVA [\Gamma] \cap \VA [\Gamma_{X_i}] = \KVA [\Gamma_{X_i}]$ (see Lemma~\ref{lem5_6B}).
Similarly, $k_j^{-1} b_{i,j} \in \KVA [\Gamma] \cap \VA [\Gamma_{X_j}] = \KVA [\Gamma_{X_j}]$.
So, 
\[
b_{i,j} \in k_i \KVA [\Gamma_{X_i}] \cap k_j \KVA [\Gamma_{X_j}] =
k_i A [\hat \Gamma_{\Phi[\Gamma_{X_i}]}] \cap k_j A [\hat \Gamma_{\Phi[\Gamma_{X_j}]}]
\]
(see Lemma \ref{lem5_6}).

By Proposition~\ref{prop5_4}\,(2), it follows that $\bigcap_{i=1}^p k_i \KVA [\Gamma_{X_i}] \neq \emptyset$, hence $\bigcap_{i=1}^p k_i \VA [\Gamma_{X_i}] \neq \emptyset$.
Choose $k \in \bigcap_{i=1}^p k_i \VA [\Gamma_{X_i}]$, and set $h = \iota_K(v)\, k$.
Then $h \in \bigcap_{i=1}^p g_i \VA [\Gamma_{X_i}]$.
\end{proof}

Let $\Sigma_1$ and $\Sigma_2$ be two abstract simplicial complexes.
Let $V_1$ (resp. $V_2$) denote the vertex set of $\Sigma_1$ (resp. $\Sigma_2$).
Recall that the \emph{join} of $\Sigma_1$ and $\Sigma_2$ is the abstract simplicial complex $\Sigma_1 \ast \Sigma_2$ defined by:
\begin{itemize}
\item
$V_1 \sqcup V_2$ is the vertex set of $\Sigma_1 \ast \Sigma_2$;
\item
a finite subset $\AA$ of $V_1 \sqcup V_2$ is a simplex of $\Sigma_1 \ast \Sigma_2$ if and only if $\AA \cap V_1$ is a simplex of $\Sigma_1$ and $\AA \cap V_2$ is a simplex of $\Sigma_2$.
\end{itemize}
Recall that $\emptyset$ is always regarded as a simplex of an abstract simplicial complex (although it is not a simplex of its geometric realization).
Consequently, $\Sigma_1$ (resp. $\Sigma_2$) is the full subcomplex of $\Sigma_1 \ast \Sigma_2$ spanned by $V_1$ (resp. $V_2$).
It is easily verified that $\Sigma_1 \ast \Sigma_2$ is a flag complex if and only if both $\Sigma_1$ and $\Sigma_2$ are flag complexes.

The following lemma is the final ingredient needed to prove Theorem \ref{thm1_2}.

\begin{lem}\label{lem5_7}
Let $\tilde v = v(g,X)$ be a vertex of $\Xi_\DD$.
Then $\Link_{\tilde v} (\Xi_\DD) \simeq \VArt [\Gamma_X] \ast \DD^X$.
\end{lem}

\begin{proof}
We begin by recalling the description of $\Link_{\tilde v} (\Xi_\DD)$ given in Section~\ref{sec4} (see the proof of Lemma~\ref{lem4_2}).
We denote by $\LL_{\tilde v}$ the set of vertices $\tilde u$ of $\Xi_\DD$ that are joined to $\tilde v$ by an edge, that is, for which there exists a $1$-cube of $\Xi_\DD$ with endpoints $\tilde u$ and $\tilde v$.
The vertex set of $\Link_{\tilde v} (\Xi_\DD)$ is $\LL_{\tilde v}$.
Let $\EE$ be a non-empty finite subset of $\LL_{\tilde v}$.
Then $\EE$ is a simplex of $\Link_{\tilde v} (\Xi_\DD)$ if and only if there exists a cube $C = C (k, Z, Z')$ of $\Xi_\DD$ such that $\tilde v \in C$ and $C \cap \LL_{\tilde v} = \EE$.
Furthermore, as usual, the empty set is also a simplex of $\Link_{\tilde v} (\Xi_\DD)$.
Recall also that a vertex $\tilde u$ of $\Xi_\DD$ belongs to $\LL_{\tilde v}$ if and only if it is of one of the following two forms:
\begin{itemize}
\item
$\tilde u = v(h, X \setminus \{ x\})$, where $x \in X$ and $g^{-1} h \in \VA [\Gamma_X]$;
\item
$\tilde u = v(g, X \cup \{ y\})$, where $y \in S \setminus X$ and $X \cup \{ y \} \in \DD$.
\end{itemize}

Recall that $\XX = \{ \xi (h,x) \mid h \in \VA [\Gamma_X],\ x \in X\}$ is the vertex set of $\VArt [\Gamma_X]$, and that $\DD_0^X = \{ y \in S \setminus X \mid X \cup \{ y\} \in \DD\}$ is the vertex set of $\DD^X$.
We define a bijection $\Phi : \LL_{\tilde v} \to \XX \sqcup \DD_0^X$ as follows.
\begin{itemize}
\item
If $x \in X$ and $g^{-1} h \in \VA [\Gamma_X]$, then $\Phi (v(h, X \setminus \{ x\})) = \xi (g^{-1}h , x) \in \XX$.
\item
If $y \in S \setminus X$ and $X \cup \{ y \} \in \DD$, then $\Phi (v (g,X \cup \{y\})) = y \in \DD_0^X$.
\end{itemize}
To prove Lemma~\ref{lem5_7}, it suffices to show that the image under $\Phi$ of every simplex of $\Link_{\tilde v} (\Xi_\DD)$ is a simplex of $\VArt[\Gamma_X] \ast \DD^X$, and that the image under $\Phi^{-1}$ of every simplex of $\VArt[\Gamma_X] \ast \DD^X$ is a simplex of $\Link_{\tilde v} (\Xi_\DD)$.

Let $\EE$ be a simplex of $\Link_{\tilde v} (\Xi_\DD)$.
Write 
\[
\EE = \{ v(h_1, X \setminus \{x_1\}), \dots, v(h_p, X \setminus \{x_p\}), v(g, X \cup \{ y_1\}), \dots, v(g, X \cup \{ y_q\}) \}\,,
\]
where $x_1, \dots, x_p \in X$, $g^{-1} h_i \in \VA [\Gamma_X]$ for every $1 \le i \le p$, $y_1, \dots, y_q \in S \setminus X$, and $X \cup \{y_j\} \in \DD$ for every $1 \le j \le q$.
Let $C = C(k, Z, Z')$ be the cube of $\Xi_\DD$ such that $\tilde v \in C$ and $C \cap \LL_{\tilde v} = \EE$.
Since $k \VA [\Gamma_{X \setminus \{ x_i\}}] = h_i \VA [\Gamma_{X \setminus \{ x_i\}}]$ for every $1 \le i \le p$, we have $g^{-1} k \in \bigcap_{i=1}^p g^{-1} h_i \VA [\Gamma_{X \setminus \{x_i\}}]$, hence $\bigcap_{i=1}^p g^{-1} h_i \VA [\Gamma_{X \setminus \{x_i\}}] \neq \emptyset$.
So, $\AA = \{ \xi(g^{-1}h_1, x_1), \dots, \xi (g^{-1}h_p, x_p)\}$ is a simplex of $\VArt [\Gamma_X]$.
Moreover, $\BB = \{y_1, \dots, y_q\}$ is a simplex of $\DD^X$ because, by Lemma~\ref{lem4_1}, $Z' = X \cup \BB \in \DD$.
It follows that $\Phi (\EE) = \AA \cup \BB$ is a simplex of $\VArt[\Gamma_X] \ast \DD^X$.

Let $\AA = \{ \xi (h_1, x_1), \dots, \xi (h_p, x_p) \}$ be a simplex of $\VArt [\Gamma_X]$, and let $\BB = \{y_1, \dots, y_q\}$ be a simplex of $\DD^X$.
Assume first that $\AA \neq \emptyset$.
By definition, $\bigcap_{i=1}^p h_i \VA [\Gamma_{X \setminus \{x_i\}}] \neq \emptyset$.
Choose $k \in \bigcap_{i=1}^p h_i \VA [\Gamma_{X \setminus \{x_i\}}]$, and set $C = C(gk,Z,Z')$, where $Z = X \setminus \{x_1, \dots, x_p\}$ and $Z' = X \cup \BB = X \cup \{y_1, \dots, y_q\}$.
The set $C$ is indeed a cube of $\Xi_\DD$, since $Z' \in \DD$ by the definition of $\DD^X$.
Since $k \in \VA [\Gamma_X]$, we have $\tilde v = v(g,X) = v(gk,X) \in C$.
Let $i \in \{1, \dots, p\}$.
Since $k^{-1}h_i \in \VA [\Gamma_{X \setminus \{x_i\}}]$, we have $gk \VA [\Gamma_{X \setminus \{x_i\}}] = gh_i \VA [\Gamma_{X \setminus \{x_i\}}]$, hence $v(gh_i, X \setminus \{x_i\}) = v(gk, X \setminus \{x_i\})$. 
On the other hand, for each $1 \le j \le q$, we have $k \in \VA [\Gamma_X] \subseteq \VA [\Gamma_{X \cup \{y_j\}}]$, hence $v(g, X \cup \{y_j\}) =  v(gk, X \cup \{y_j\})$.
It follows that $C \cap \LL_{\tilde v} =  \{ v(gh_1, X \setminus \{x_1\}), \dots, v(gh_p, X \setminus \{x_p\}), v(g, X \cup \{ y_1\}), \dots, v(g, X \cup \{ y_q\}) \} = \Phi^{-1} (\AA \cup \BB)$, hence $\Phi^{-1} (\AA \cup \BB)$ is a simplex of $\Link_{\tilde v} (\Xi_\DD)$.

Assume now that $\AA = \emptyset$.
Let $C= C (g,X,Z')$, where $Z' = X \cup \BB$.
Since $\BB$ is a simplex of $\DD^X$, we have $Z' \in \DD$, hence $C$ is indeed a cube of $\Xi_\DD$.
It is easy to see that $\tilde v = v(g,X) \in C$ and that $C \cap \LL_{\tilde v} = \{ v(g, X \cup \{ y_1\}), \dots, v(g, X \cup \{ y_q\}) \} = \Phi^{-1} (\BB)$, hence $\Phi^{-1} (\BB)$ is a simplex of $\Link_{\tilde v} (\Xi_\DD)$.
\end{proof}

\begin{proof}[Proof of Theorem \ref{thm1_2}]
By Theorem \ref{thm1_1}, the cube complex $\Xi_\DD$ is connected, simply connected, and locally regular.
By Theorem~\ref{thm5_1}, it follows that $\Xi_\DD$ is CAT(0) if and only if $\Link_{\tilde v} (\Xi_\DD)$ is a flag complex for every vertex $\tilde v$ of $\Xi_\DD$.
Let $\tilde v = v(g,X)$ be a vertex of $\Xi_\DD$.
By Lemma~\ref{lem5_7},  $\Link_{\tilde v} (\Xi_\DD) \simeq \VArt [\Gamma_X] \ast \DD^X$, and, by Corollary~\ref{corl5_3},  $\VArt [\Gamma_X]$ is a flag complex, hence $\Link_{\tilde v} (\Xi_\DD)$ is a flag complex if and only if $\DD^X$ is a flag complex.
So, $\Xi_\DD$ is a CAT(0) cube complex if and only if $\DD^X$ is a flag complex for every $X \in \DD$.

If $\DD^X$ is a flag complex for every $X \in \DD$, then, in particular, $\DD = \DD^\emptyset$ is a flag complex.
Conversely, suppose that $\DD$ is a flag complex.
Let $X \in \DD$.
Let $\AA$ be a finite subset of $\DD_0^X = \{ y \in S \setminus X \mid X \cup \{ y \} \in \DD\}$ such that, for all distinct $y_1, y_2 \in \AA$, the set $\{y_1, y_2\}$ is a simplex of $\DD^X$, that is, $X \cup \{y_1, y_2\}$ is a simplex of $\DD$.
Then $\{z_1, z_2\} \in \DD$ for all distinct $z_1, z_2 \in X \cup \AA$.
Since $\DD$ is a flag complex, this implies that $X \cup \AA$ is a simplex of $\DD$, and therefore that $\AA$ is a simplex of $\DD^X$.
So, $\DD^X$ is a flag complex.
We conclude that $\DD^X$ is a flag complex for every $X \in \DD$ if and only if $\DD$ is a flag complex.
\end{proof}


\begin{fund}
The first author is part of the research project PID2022-138719NA-I00, funded by\\
MCIN/AEI/10.13039/501100011033/FEDER, UE. The second author is partially supported by the French National Research Agency (ANR) through the project CaGeT (ANR-25-CE40-4162). The Institut de Mathématiques de Bourgogne (IMB) receives support from the EIPHI Graduate School (grant ANR-17-EURE-0002).
\end{fund}



\end{document}